\documentclass[12pt]{article}
\usepackage{amsmath,amssymb,amsthm}
\newtheorem{thm}{Theorem}[section]

 \newtheorem{cor}{Corollary}[section]
 \newtheorem{lem}{Lemma}[section]
 \newtheorem{prop}{Proposition}[section]
 \newtheorem{defn}{Definition}[section]%是否是全文计数？
\newtheorem{rem}{Remark}[section]

\begin{document}
\begin{center}
{\large{\bf Global existence and minimal decay regularity for the Timoshenko system: The case of non-equal wave speeds}}
\end{center}
\begin{center}
\footnotesize{Jiang Xu}\\[2ex]
\footnotesize{Department of Mathematics, \\ Nanjing
University of Aeronautics and Astronautics, \\
Nanjing 211106, P.R.China,}\\
\footnotesize{jiangxu\underline{ }79@nuaa.edu.cn}\\
\vspace{10mm}

\footnotesize{Naofumi Mori}\\[2ex]
\footnotesize{Graduate School of Mathematics,\\ Kyushu University, Fukuoka 819-0395, Japan,}\\
\footnotesize{n-mori@math.kyushu-u.ac.jp}\\

\vspace{10mm}

\footnotesize{Shuichi Kawashima}\\[2ex]
\footnotesize{Faculty of Mathematics, \\ Kyushu University, Fukuoka 819-0395, Japan,}\\
\footnotesize{kawashim@math.kyushu-u.ac.jp}\\
\end{center}
\vspace{6mm}

\begin{abstract}
As a continued work of \cite{MXK}, we are concerned with the Timoshenko system in the case of non-equal wave speeds, which admits
the dissipative structure of \textit{regularity-loss}. Firstly, with the modification of a priori estimates in \cite{MXK},
we construct global solutions to the Timoshenko system pertaining to data in the Besov space with the regularity $s=3/2$. Owing to the weaker dissipative mechanism, extra higher regularity than that for the global-in-time existence is usually imposed to obtain the optimal decay rates of classical solutions, so
it is almost impossible to obtain the optimal decay rates in the critical space. To overcome the outstanding difficulty, we develop a new frequency-localization time-decay inequality,
which captures the information related to the integrability at the high-frequency part. Furthermore, by the energy approach in terms of high-frequency and low-frequency decomposition, we show the optimal decay rate for Timoshenko system in critical Besov spaces, which improves previous works greatly.
\end{abstract}


\section{Introduction}
In this work, we are concerned with the following Timoshenko system (see \cite{T1,T2}), which is a set of two coupled wave equations of the form
\begin{align}
\left\{\begin{array}{l}
       \varphi_{tt}-(\varphi_x-\psi)_x=0,\\[2mm]
       \psi_{tt}-\sigma(\psi_{x})_{x}-(\varphi_x-\psi)+\gamma \psi_t =0.
       \end{array}\right. \label{R-E1}
\end{align}
System (\ref{R-E1}) describes the transverse vibrations of a beam. Here, $t\geq 0$ is the time variable, $x\in \mathbb{R}$ is the spacial
variable which denotes the point on the center line of the beam,
$\varphi(t,x)$ is the transversal displacement of the beam from an equilibrium state, and $\psi(t,x)$ is the rotation
angle of the filament of the beam. The smooth function $\sigma(\eta)$ satisfies $\sigma'(\eta)>0$ for any $\eta\in\mathbb{R}$, and $\gamma$
is a positive constant. We focus on the Cauchy problem of (\ref{R-E1}), so the initial data are supplemented as
\begin{equation}
(\varphi, \varphi_{t}, \psi, \psi_{t})(x,0)
=(\varphi_{0}, \varphi_{1}, \psi_{0}, \psi_{1})(x).\label{R-E2}
\end{equation}
Based on the change of variable introduced by Ide, Haramoto, and the third author \cite{IHK}:
\begin{align}\label{R-E3}
v=\varphi_x-\psi, \quad
u=\varphi_t, \quad
z=a\psi_x, \quad
y=\psi_t,
\end{align}
with $a>0$ being the sound speed defined by $a^2=\sigma'(0)$, it is convenient to rewrite (\ref{R-E1})-(\ref{R-E2})
as a Cauchy problem for the first-order hyperbolic system of $U=(v,u,z,y)^{\top}$
\begin{align}\label{R-E4}
\left\{\begin{array}{l}
        U_t + A(U)U_x + LU =0,\\[2mm]
        U(x, 0) = U_0(x)
\end{array}\right.
\end{align}
with $U_{0}(x)=(v_{0}, u_{0}, z_{0}, y_{0})(x)$, where $v_0=\varphi_{0,x}-\psi_{0}$,
$u_0=\varphi_1$, $z_0=a\psi_{0,x}$, $y_0=\psi_1$ and
\begin{align*}
A(U)=-\left(
\begin{array}{cccc}
0 & 1 & 0 & 0 \\
1 & 0 & 0 & 0 \\
0 & 0 & 0 & a \\
0 & 0 &\frac{\sigma'(z/a)}{a} & 0
\end{array}
\right),\ \ \
L=\left(
\begin{array}{cccc}
0 & 0 & 0 & 1 \\
0 & 0 & 0 & 0 \\
0 & 0 & 0 & 0 \\
-1 & 0 & 0 & \gamma
\end{array}
\right).
\end{align*}
Note that $A(U)$ is a real symmetrizable matrix due to $\sigma'(z/a)>0$, and the dissipative matrix $L$ is nonnegative definite but not symmetric. Such degenerate dissipation forces (\ref{R-E4}) to go beyond
the class of generally dissipative hyperbolic systems, so the recent global-in-time existence (see \cite{XK1}) for hyperbolic systems with symmetric dissipation
can not be applied directly, which is the main motivation on studying the Timoshenko system (\ref{R-E1}).

Let us review several known results on (\ref{R-E1}). In a bounded domain, it is known that (\ref{R-E1}) is exponentially stable if the damping term $\varphi_{t}$ is also present on the left-hand side of the first equation of (\ref{R-E3}) (see, e.g., \cite{RFSC}). Soufyane \cite{S} showed that (\ref{R-E1}) could not be exponentially stable by considering only the damping term of the form $\psi_{t}$, unless for the case of $a=1$ (equal wave speeds). A similar result was obtained by Rivera and Racke \cite{RR2} with an alternative proof. In addition, Rivera and Racke \cite{RR1} also investigated
the Timoshenko system with the heat conduction, which is described by the classical Fourier law. In the whole space, the third author and his collaborators \cite{IHK} considered the corresponding linearized form of
 (\ref{R-E4}):
\begin{equation}\label{R-E5}
\left\{\begin{array}{l}
        v_t-u_x+y=0,\\[2mm]
        u_t-v_x=0,\\[2mm]
        z_t-ay_x=0,\\[2mm]
         y_t-az_x-v+\gamma y=0,\\[2mm]
         (v, u, z, y)(x, 0)
=(v_{0}, u_{0}, z_{0}, y_{0})(x),
\end{array}\right.
\end{equation}
and showed that the dissipative structure
could be characterized by
\begin{equation}
\left\{\begin{array}{l}
{\rm Re}\,\lambda(i\xi)\leq -c\eta_1(\xi)
   \qquad {\rm for} \quad a=1, \\[1mm]
{\rm Re}\,\lambda(i\xi)\leq -c\eta_2(\xi)
   \qquad {\rm for} \quad a\neq 1,
\end{array}\right.\label{R-E6}
\end{equation}
where $\lambda(i\xi)$ denotes the eigenvalues of the system
(\ref{R-E5}) in the Fourier space, $\eta_1(\xi)=\frac{\xi^2}{1+\xi^2}$,
$\eta_2(\xi)=\frac{\xi^2}{(1+\xi^2)^2}$, and $c>0$ is some
constant. Consequently, the following decay properties were established
for $U=(v,u,z,y)^{\top}$ of (\ref{R-E5}) (see \cite{IHK} for details):
\begin{equation}\label{R-E7}
\|\partial_x^k U(t)\|_{L^2}
\lesssim
(1+t)^{-\frac{1}{4}-\frac{k}{2}}\|U_0\|_{L^1}
+e^{-ct}\|\partial_x^kU_0\|_{L^2}
\end{equation}
for $a=1$, and
\begin{equation}\label{R-E8}
\|\partial_x^k U(t)\|_{L^2}
\lesssim
(1+t)^{-\frac{1}{4}-\frac{k}{2}}\|U_0\|_{L^1}
+(1+t)^{-\frac{l}{2}}\|\partial_x^{k+l}U_0\|_{L^2}
\end{equation}
for $a\neq 1$. Recently, under the additional assumption $\int_{\mathbb{R}}U_{0}dx=0$, Racke and Said-Houari \cite{RS} strengthened  (\ref{R-E7})-(\ref{R-E8}) such that linearized solutions
decay faster with a rate of $t^{-\gamma/2}$, by introducing the integral space $L^{1,\gamma}(\mathbb{R})$.

\begin{rem}
Clearly, the high frequency part of (\ref{R-E7}) yields an exponential decay, whereas the corresponding part of (\ref{R-E8}) is of the regularity-loss type,  since $(1+t)^{-\ell/2}$ is created by assuming the additional $\ell$-th order regularity on the initial data. Consequently, extra higher regularity than that for global-in-time existence of classical solutions
is imposed to obtain the optimal decay rates.
\end{rem}

In \cite{IK}, Ide and the third author performed the time-weighted approach to establish the global existence
and asymptotic decay of solutions to the nonlinear problem (\ref{R-E4}). To overcome the difficulty caused by the regularity-loss property,
the spatially regularity $s\geq6$ was needed. Denote by $s_{c}$ the critical regularity for global existence of classical solutions.
Actually, the local-in-time existence theory of Kato and Majda \cite{K,M} implies that $s_{c}=2$ for the Timoshenko system (\ref{R-E4}),
actually, the
extra regularity is used to take care of optimal decay estimates. Consequently, some natural questions follow. Is $s=6$ the minimal decay regularity for
(\ref{R-E4}) with the regularity-loss? If not, which index characterises
the minimal decay regularity? This motivates the following general definition.
\begin{defn}\label{defn1.1}
If the optimal decay rate of $L^{1}(\mathbb{R}^n)$-$L^2(\mathbb{R}^n)$ type is achieved under the lowest regularity assumption, then the lowest index is called the minimal decay regularity index for dissipative systems of regularity-loss, which is labelled as $s_{D}$.
\end{defn}

Recently, we are concerned with the global existence and large-time behavior for (\ref{R-E4})
in spatially critical Besov spaces. To the best of our knowledge, there are few results available in this direction for the Timoshenko system, although
the critical space has already been succeeded in the study of fluid dynamical equations, see \cite{AGZ,D1,H,PZ} for Navier-Stokes equations,
\cite{D2,XW,XY,Z} for Euler equations and related models. In \cite{XK1,XK2}, under the assumptions of dissipative entropy and Shizuta-Kawashima condition, the first and third authors have already investigated generally dissipative systems, however, the Timoshenko system admits the non-symmetric dissipation and goes beyond the class. Hence, as a first step, we \cite{MXK} considered the easy case, that is, (\ref{R-E4}) with the equal wave speed ($a=1$).
By virtue of an elementary
fact in Proposition \ref{prop2.3} (also see \cite{XK1}) that indicates the relation between homogeneous and inhomogeneous Chemin-Lerner
spaces, we first constructed global solutions pertaining to data
in the Besov space $B^{3/2}_{2,1}(\mathbb{R})$. Furthermore, the optimal decay rates
of solution and its derivatives are shown in the space $B^{3/2}_{2,1}(\mathbb{R})\cap \dot{B}^{-1/2}_{2,\infty}(\mathbb{R})$
by the frequency-localization Duhamel principle and energy approach in terms of high-frequency and low-frequency decomposition.

In the present paper, we hope to establish similar results for (\ref{R-E4}) with non-equal wave speeds ($a\neq1$) that has weaker dissipative mechanism. If done, we shall improve two regularity indices for  Timoshenko system with regularity-loss: $s_{c}=3/2$ for global-in-time existence and $s_{D}=3/2$ for the optimal decay estimate, which lead to
reduce significantly the regularity requirements on the initial data in comparison with \cite{IK}.

Before main results, let us explain new technical points for (\ref{R-E4}) with $a\neq1$ and the strategy to get round the obstruction.
Firstly, as in \cite{MXK}, the degenerate non-symmetric damping enables us to capture the dissipation from contributions of
$(y,v,u_{x},z_{x})$, however, there is an additional norm related to $u_{x}$ in the proof for the dissipation of $v$. Indeed, we need to
 carefully take care of the topological relation between $\|u_{x}\|_{\widetilde{L}^{2}_{T}(\dot{B}^{-1/2}_{2,1})}$ and $\|u_{x}\|_{\widetilde{L}^{2}_{T}(B^{-1/2}_{2,1})}$ as in Proposition \ref{prop2.3}. To do this, we localize (\ref{R-E4}) with inhomogeneous blocks rather than homogeneous blocks to
obtain the dissipative estimate for $v$.

Secondly, due to the weaker mechanism of regularity-loss, it seems that
there is no possibility to capture optimal decay rates in the critical space $B^{3/2}_{2,1}(\mathbb{R})$, since
the polynomial decay at the high-frequency part comes from the fact that the initial data is imposed
extra higher regularity (see (\ref{R-E8})). To overcome the outstanding difficulty, there are new ingredients in comparison with the case of equal wave speeds in \cite{MXK}. Precisely, we develop a new frequency-localization time-decay
inequality for the dissipative rate $\eta(\xi)=\frac{|\xi|^2}{(1+|\xi|^2)^2}$ in $\mathbb{R}^{n}$, see Proposition \ref{prop3.1}. At the formal level, we see that the high-frequency part decays in time not only with algebraic rates of any order as long as the function is spatially regular enough, but also additional information related the $L^p$-integrability is available. Consequently, the high-frequency estimate in energy approaches
can be divided into two parts, and on each part, different values of $p$ (for example, $p=1$ or $p=2)$ are chosen to get desired decay estimates, see Lemma \ref{lem5.1}. Additionally, it should be worth noting that the energy approach is totally different from that in \cite{MXK}, where the frequency-localization Duhamel principle was used. Here, we shall employ somewhat ``the square formula of the Duhamel principle" based on the Littlewood-Paley pointwise estimate in
Fourier space for the linear system with right-hand side, see (\ref{R-E36})-(\ref{R-E37}) for details.

Our main results focus on the Timoshenko system with non-equal wave speeds ($a\neq1$), which are stated as follows.
\begin{thm}\label{thm1.1}
Suppose that $U_{0}\in B^{3/2}_{2,1}(\mathbb{R})$. There exists a positive constant $\delta_{0}$ such that if
$$\|U_{0}\|_{B^{3/2}_{2,1}(\mathbb{R})}\leq
\delta_{0}, $$
then the Cauchy problem (\ref{R-E4}) has a unique
global classical solution $U\in \mathcal{C}^{1}(\mathbb{R}^{+}\times
\mathbb{R})$ satisfying
$$U \in\widetilde{\mathcal{C}}(B^{3/2}_{2,1}(\mathbb{R}))\cap\widetilde{\mathcal{C}}^{1}(B^{1/2}_{2,1}(\mathbb{R}))$$
Moreover, the following energy inequality holds that
\begin{eqnarray}
&&\|U\|_{\widetilde{L}^\infty(B^{3/2}_{2,1}(\mathbb{R}))}+\Big(\|y\|_{\widetilde{L}^2_{T}(B^{3/2}_{2,1})}+\|(v,z_{x})\|_{\widetilde{L}^2_{T}(B^{1/2}_{2,1})}
+\|u_{x}\|_{\widetilde{L}^2_{T}(B^{-1/2}_{2,1})}\Big)\nonumber\\&\leq&  C_{0}\|U_{0}\|_{B^{3/2}_{2,1}(\mathbb{R})}, \label{R-E9}
\end{eqnarray}
where $C_{0}>0$ is a constant.
\end{thm}

\begin{rem}
Theorem \ref{thm1.1} exhibits the optimal critical regularity ($s_{c}=3/2$) of global-in-time existence for (\ref{R-E4}), which was proved by the revised energy estimates in comparison with \cite{MXK}, along with the local-in-time existence result in Proposition \ref{prop4.1}. Observe that there is 1-regularity-loss phenomenon for the dissipation rate of $(v,u_{x})$.
\end{rem}

Furthermore, with the aid of the new frequency-localization time-decay
inequality in Proposition \ref{prop3.1}, we can obtain the the optimal decay estimates by using the time-weighted energy approach in terms of high-frequency and
low-frequency decomposition.

\begin{thm}\label{thm1.2}
Let $U(t,x)=(v,u,z,y)(t,x)$ be the global classical solution of Theorem \ref{thm1.1}.
Assume that the initial data satisfy $U_{0}\in B^{3/2}_{2,1}(\mathbb{R})\cap\dot{B}^{-1/2}_{2,\infty}(\mathbb{R})$.
Set $I_{0}:=\|U_{0}\|_{B^{3/2}_{2,1}(\mathbb{R})\cap\dot{B}^{-1/2}_{2,\infty}(\mathbb{R})}$. If $I_{0}$ is sufficiently small,
then the classical solution $U(t,x)$ of (\ref{R-E4}) admits the optimal decay estimate
\begin{eqnarray}
\|U\|_{L^2}\lesssim I_{0}(1+t)^{-\frac{1}{4}}. \label{R-E10}
\end{eqnarray}
\end{thm}

Note that the embedding $L^1(\mathbb{R})\hookrightarrow \dot{B}^{-1/2}_{2,\infty}(\mathbb{R})$ in Lemma \ref{lem2.3}, as an immediate byproduct of Theorem \ref{thm1.2}, the usual optimal decay estimate of $L^{1}(\mathbb{R})$-$L^{2}(\mathbb{R})$ type is available.

\begin{cor}\label{cor1.1}
Let $U(t,x)=(v,u,z,y)(t,x)$ be the global classical solutions of Theorem \ref{thm1.1}. If further the initial data
$U_{0}\in L^1(\mathbb{R})$ and $\widetilde{I}_{0}:=\|U_{0}\|_{B^{3/2}_{2,1}(\mathbb{R})\cap L^1(\mathbb{R})}$ is sufficiently small,
then
\begin{eqnarray}
\|U\|_{L^2}\lesssim \widetilde{I}_{0}(1+t)^{-\frac{1}{4}}. \label{R-E11}
\end{eqnarray}
\end{cor}

\begin{rem}
Let us mention that Theorem \ref{thm1.2} and Corollary \ref{cor1.1} exhibit the optimal decay rate in the Besov space with $s_{c}=3/2$,
that is, $s_{D}=3/2$, which implies that the minimal decay regularity coincides with the the critical regularity for global solutions,
and the extra higher regularity is not necessary. In addition, it is worth noting that the present work opens a door for the study of dissipative systems of regularity-loss type, which encourages us to develop frequency-localization time-decay
inequalities for other dissipative rates and investigate systems with the regularity-loss mechanism.
\end{rem}

Finally, we would like to mention other studies on the dissipative Timoshenko system with different effects, see, e.g.,
\cite{RBS,RFSC} for frictional dissipation case, \cite{FR,SAJM,SK} for thermal dissipation case, and \cite{ABMR,ARMSV,LK,LP} for memory-type dissipation case.

The rest of this paper unfolds as follows. In Sect.\ref{sec:2}, we
present useful properties in Besov spaces, which will be used in the subsequence
analysis. In Sect.\ref{sec:3}, we shall develop new time-decay
inequality with using frequency-localization techniques. Sect.\ref{sec:4} is devoted to construct the global-in-time existence of classical solutions to (\ref{R-E4}). Furthermore, in Sect.\ref{sec:5}, we deduce the optimal decay estimate for (\ref{R-E4})
by employing energy approaches in terms of high-frequency and low-frequency decomposition.
In Appendix (Sect.\ref{sec:6}), we present those definitions for Besov spaces and
Chemin-Lerner spaces for the convenience of reader.

\textbf{Notations.}  Throughout the paper, $f\lesssim g$ denotes $f\leq Cg$, where $C>0$
is a generic constant. $f\thickapprox g$ means $f\lesssim g$ and $g\lesssim f$. Denote by $\mathcal{C}([0,T],X)$ (resp.,
$\mathcal{C}^{1}([0,T],X)$) the space of continuous (resp.,
continuously differentiable) functions on $[0,T]$ with values in a
Banach space $X$. Also, $\|(f,g,h)\|_{X}$ means $
\|f\|_{X}+\|g\|_{X}+\|h\|_{X}$, where $f,g,h\in X$.

\section{Tools}\setcounter{equation}{0}\label{sec:2}
In this section, we only collect useful analysis properties in Besov spaces and Chemin-Lerner spaces in $\mathbb{R}^{n}(n\geq1)$.
For convenience of reader, those definitions for Besov spaces and Chemin-Lerner spaces are given in the Appendix.
Firstly, we give an improved
Bernstein inequality (see, \textit{e.g.}, \cite{W}), which allows the case of fractional derivatives.

\begin{lem}\label{lem2.1}
Let $0<R_{1}<R_{2}$ and $1\leq a\leq b\leq\infty$.
\begin{itemize}
\item [(i)] If $\mathrm{Supp}\mathcal{F}f\subset \{\xi\in \mathbb{R}^{n}: |\xi|\leq
R_{1}\lambda\}$, then
\begin{eqnarray*}
\|\Lambda^{\alpha}f\|_{L^{b}}
\lesssim \lambda^{\alpha+n(\frac{1}{a}-\frac{1}{b})}\|f\|_{L^{a}}, \ \  \mbox{for any}\ \  \alpha\geq0;
\end{eqnarray*}

\item [(ii)]If $\mathrm{Supp}\mathcal{F}f\subset \{\xi\in \mathbb{R}^{n}:
R_{1}\lambda\leq|\xi|\leq R_{2}\lambda\}$, then
\begin{eqnarray*}
\|\Lambda^{\alpha}f\|_{L^{a}}\approx\lambda^{\alpha}\|f\|_{L^{a}}, \ \  \mbox{for any}\ \ \alpha\in\mathbb{R}.
\end{eqnarray*}
\end{itemize}
\end{lem}

Besov spaces obey various inclusion relations. Precisely,
\begin{lem}\label{lem2.2} Let $s\in \mathbb{R}$ and $1\leq
p,r\leq\infty,$ then
\begin{itemize}
\item[(1)]If $s>0$, then $B^{s}_{p,r}=L^{p}\cap \dot{B}^{s}_{p,r};$
\item[(2)]If $\tilde{s}\leq s$, then $B^{s}_{p,r}\hookrightarrow
B^{\tilde{s}}_{p,r}$. This inclusion relation is false for
the homogeneous Besov spaces;
\item[(3)]If $1\leq r\leq\tilde{r}\leq\infty$, then $\dot{B}^{s}_{p,r}\hookrightarrow
\dot{B}^{s}_{p,\tilde{r}}$ and $B^{s}_{p,r}\hookrightarrow
B^{s}_{p,\tilde{r}};$
\item[(4)]If $1\leq p\leq\tilde{p}\leq\infty$, then $\dot{B}^{s}_{p,r}\hookrightarrow \dot{B}^{s-n(\frac{1}{p}-\frac{1}{\tilde{p}})}_{\tilde{p},r}
$ and $B^{s}_{p,r}\hookrightarrow
B^{s-n(\frac{1}{p}-\frac{1}{\tilde{p}})}_{\tilde{p},r}$;
\item[(5)]$\dot{B}^{n/p}_{p,1}\hookrightarrow\mathcal{C}_{0},\ \ B^{n/p}_{p,1}\hookrightarrow\mathcal{C}_{0}(1\leq p<\infty);$
\end{itemize}
where $\mathcal{C}_{0}$ is the space of continuous bounded functions
which decay at infinity.
\end{lem}

\begin{lem}\label{lem2.3}
Suppose that $\varrho>0$ and $1\leq p<2$. It holds that
\begin{eqnarray*}
\|f\|_{\dot{B}^{-\varrho}_{r,\infty}}\lesssim \|f\|_{L^{p}}
\end{eqnarray*}
with $1/p-1/r=\varrho/n$. In particular, this holds with $\varrho=n/2, r=2$ and $p=1$.
\end{lem}

Moser-type product estimates are stated as follows, which plays an important role in the estimate of bilinear
terms.
\begin{prop}\label{prop2.1}
Let $s>0$ and $1\leq
p,r\leq\infty$. Then $\dot{B}^{s}_{p,r}\cap L^{\infty}$ is an algebra and
$$
\|fg\|_{\dot{B}^{s}_{p,r}}\lesssim \|f\|_{L^{\infty}}\|g\|_{\dot{B}^{s}_{p,r}}+\|g\|_{L^{\infty}}\|f\|_{\dot{B}^{s}_{p,r}}.
$$
Let $s_{1},s_{2}\leq n/p$ such that $s_{1}+s_{2}>n\max\{0,\frac{2}{p}-1\}. $  Then one has
$$\|fg\|_{\dot{B}^{s_{1}+s_{2}-n/p}_{p,1}}\lesssim \|f\|_{\dot{B}^{s_{1}}_{p,1}}\|g\|_{\dot{B}^{s_{2}}_{p,1}}.$$
\end{prop}

In the analysis of decay estimates, we also need the general form of Moser-type product estimates, which was shown by Yong in \cite{Z}.
\begin{prop}\label{prop2.2}
Let $s>0$ and $1\leq p,r,p_{1},p_{2},p_{3},p_{4}\leq\infty$. Assume
that $f\in L^{p_{1}}\cap \dot{B}^{s}_{p_{4},r}$ and $g\in L^{p_{3}}\cap
\dot{B}^{s}_{p_{2},r}$ with
$$\frac{1}{p}=\frac{1}{p_{1}}+\frac{1}{p_{2}}=\frac{1}{p_{3}}+\frac{1}{p_{4}}.$$
Then it holds that
\begin{eqnarray*}
\|fg\|_{\dot{B}^{s}_{p,r}}\lesssim \|f\|_{L^{p_{1}}}\|g\|_{\dot{B}^{s}_{p_{2},r}}+\|g\|_{L^{p_{3}}}\|f\|_{\dot{B}^{s}_{p_{4},r}}.
\end{eqnarray*}
\end{prop}

In \cite{XK1}, the first and third authors established a key fact, which indicates the connection between
homogeneous Chemin-Lerner spaces and inhomogeneous Chemin-Lerner spaces.
\begin{prop} \label{prop2.3}
Let $s\in \mathbb{R}$ and $1\leq \theta, p,r\leq\infty$.
\begin{itemize}
\item [$(1)$] It holds that
\begin{eqnarray*}
L^{\theta}_{T}(L^{p})\cap
\widetilde{L}^{\theta}_{T}(\dot{B}^{s}_{p,r})\subset \widetilde{L}^{\theta}_{T}(B^{s}_{p,r});
\end{eqnarray*}
\item [$(2)$] Furthermore, as $s>0$ and $\theta\geq r$, it holds that
\begin{eqnarray*}
L^{\theta}_{T}(L^{p})\cap
\widetilde{L}^{\theta}_{T}(\dot{B}^{s}_{p,r})=\widetilde{L}^{\theta}_{T}(B^{s}_{p,r})
\end{eqnarray*}
\end{itemize}
for any $T>0$.
\end{prop}

The property of continuity for product in $\widetilde{L}^{\theta}_{T}(B^{s}_{p,r})$ is similar to in the stationary case (Proposition \ref{prop2.1}),
whereas the time exponent $\theta$ behaves according to the H\"{o}lder inequality.
\begin{prop}\label{prop2.4}
The following inequality holds:
$$
\|fg\|_{\widetilde{L}^{\theta}_{T}(B^{s}_{p,r})}\lesssim
(\|f\|_{L^{\theta_{1}}_{T}(L^{\infty})}\|g\|_{\widetilde{L}^{\theta_{2}}_{T}(B^{s}_{p,r})}
+\|g\|_{L^{\theta_{3}}_{T}(L^{\infty})}\|f\|_{\widetilde{L}^{\theta_{4}}_{T}(B^{s}_{p,r})})
$$
whenever $s>0, 1\leq p\leq\infty,
1\leq\theta,\theta_{1},\theta_{2},\theta_{3},\theta_{4}\leq\infty$
and
$$\frac{1}{\theta}=\frac{1}{\theta_{1}}+\frac{1}{\theta_{2}}=\frac{1}{\theta_{3}}+\frac{1}{\theta_{4}}.$$
As a direct corollary, one has
$$\|fg\|_{\widetilde{L}^{\theta}_{T}(B^{s}_{p,r})}
\lesssim
\|f\|_{\widetilde{L}^{\theta_{1}}_{T}(B^{s}_{p,r})}\|g\|_{\widetilde{L}^{\theta_{2}}_{T}(B^{s}_{p,r})}$$
whenever $s\geq n/p,
\frac{1}{\theta}=\frac{1}{\theta_{1}}+\frac{1}{\theta_{2}}.$
\end{prop}
Finally, we state a continuity result for compositions (see
\cite{Abidi}) to end this section.
\begin{prop}\label{prop2.5}
Let $s>0$, $1\leq p, r, \rho\leq \infty$, $F\in
W^{[s]+1,\infty}_{loc}(I;\mathbb{R})$ with $F(0)=0$, $T\in
(0,\infty]$ and $v\in \widetilde{L}^{\rho}_{T}(B^{s}_{p,r})\cap
L^{\infty}_{T}(L^{\infty}).$ Then
$$\|F(v)\|_{\widetilde{L}^{\rho}_{T}(B^{s}_{p,r})}\lesssim
(1+\|v\|_{L^{\infty}_{T}(L^{\infty})})^{[s]+1}\|v\|_{\widetilde{L}^{\rho}_{T}(B^{s}_{p,r})}.$$
\end{prop}

\section{Frequency-localization time-decay inequality}\setcounter{equation}{0}\label{sec:3}
In the recent decade, harmonic analysis tools, especially for techniques based on Littlewood-Paley decomposition and paradifferential calculus
have proved to be very efficient in the study of partial differential equations. It is well-known that the frequency-localization operator $\dot{\Delta}_{q}f$ (or $\Delta_{q}f$ )
has a smoothing effect on the function $f$, even though $f$ is quite rough. Moreover, the $L^p$ norm of $\dot{\Delta}_{q}f$ can be preserved provided $f\in L^p(\mathbb{R}^{n})$. To the best of our knowledge, so far there are few efforts about the decay property related to the operator $\dot{\Delta}_{q}f$. Here, the difficulty of regularity-loss mechanism forces us to develop the frequency-localization time-decay inequality. Precisely,

\begin{prop}\label{prop3.1} Set $\eta(\xi)=\frac{\mid\xi\mid^{2}}{(1+\mid\xi\mid^{2})^{2}}$. If $f\in \dot{B}^{\sigma+\ell}_{2,r}(\mathbb{R}^{n})\cap \dot{B}^{-s}_{2,\infty}(\mathbb{R}^{n})$ for $\sigma\in \mathbb{R}, s\in \mathbb{R}$ and $1\leq r\leq\infty$ such that $\sigma+s>0$, then it holds that
\begin{eqnarray}
&&\Big\|2^{q\sigma}\|\widehat{\dot{\Delta}_{q}f}e^{-\eta(\xi)t}\|_{L^{2}}\Big\|_{l^{r}_{q}}\nonumber\\ &\lesssim & \underbrace{(1+t)^{-\frac{\sigma+s}{2}}\|f\|_{\dot{B}^{-s}_{2,\infty}}}_{Low-frequency\  Estimate}
+\underbrace{(1+t)^{-\frac{\ell}{2}+\frac{n}{2}(\frac{1}{p}-\frac{1}{2})}\|f\|_{\dot{B}^{\sigma+\ell}_{p,r}}}_{High-frequency\  Estimate}, \label{R-E12}
\end{eqnarray}
for $\ell>n(\frac{1}{p}-\frac{1}{2})$\ \footnote{Let us remark that $\ell\geq0$ in the case of $p=2$.} with $1\leq p\leq2$.
\end{prop}

\begin{proof}
For clarity, the proof is separated into high-frequency and low-frequency parts.

(1) If $q\geq 0$, then $|\xi|\sim2^{q}\geq1$, which leads to
\begin{eqnarray}
\|\widehat{\dot{\Delta}_{q}f}e^{-\eta(\xi)t}\|_{L^2}&\leq&\|\widehat{\dot{\Delta}_{q}f}e^{-c_{0}t|\xi|^{-2}}\|_{L^2(|\xi|\geq 1)}\nonumber\\&=&
\Big\||\xi|^{\ell}|\widehat{\dot{\Delta}_{q}f}|\frac{e^{-c_{0}t|\xi|^{-2}}}{|\xi|^{\ell}}\Big\|_{L^2(|\xi|\geq 1)}\nonumber\\&\leq&
\||\xi|^{\ell}\widehat{\dot{\Delta}_{q}f}\|_{L^{p'}}\Big\|\frac{e^{-c_{0}t|\xi|^{-2}}}{|\xi|^{\ell}}\Big\|_{L^{s}(|\xi|\geq 1)}
 \ \ \Big(\frac{1}{p'}+\frac{1}{s}=\frac{1}{2},\ p'\geq2 \Big)\nonumber\\&\leq&
 2^{q\ell}\|\dot{\Delta}_{q}f\|_{L^{p}}\Big\|\frac{e^{-c_{1}t|\xi|^{-2}}}{|\xi|^{\ell}}\Big\|_{L^{s}(|\xi|\geq 1)}\ \  \Big(\frac{1}{p}+\frac{1}{p'}=1\Big), \label{R-E13}
\end{eqnarray}
where $c_{1}>0$ and the Hausdorff-Young's inequality was used  in the last line. By performing the change of variable as in \cite{XMK}, we arrive at
\begin{eqnarray}
\Big\|\frac{e^{-c_{1}t|\xi|^{-2}}}{|\xi|^{\ell}}\Big\|_{L^{s}(|\xi|\geq 1)}\lesssim (1+t)^{-\frac{\ell}{2}+\frac{n}{2}(\frac{1}{p}-\frac{1}{2})} \label{R-E14}
\end{eqnarray}
for $\ell>n(\frac{1}{p}-\frac{1}{2})$. Besides, it can be also bounded by $(1+t)^{-\frac{\ell}{2}}$
for $\ell\geq0$ if $p=2$. Then it follows from (\ref{R-E13})-(\ref{R-E14}) that
\begin{eqnarray}
2^{q\sigma}\|\widehat{\dot{\Delta}_{q}f}e^{-\eta(\xi)t}\|_{L^2}\lesssim 2^{q(\sigma+\ell)}(1+t)^{-\frac{\ell}{2}+\frac{n}{2}(\frac{1}{p}-\frac{1}{2})}\|\dot{\Delta}_{q}f\|_{L^{p}}. \label{R-E15}
\end{eqnarray}

(2) If $q<0$, then $|\xi|\sim2^{q}<1$, which implies that
\begin{eqnarray}
|\widehat{\dot{\Delta}_{q}f}|e^{-\eta(\xi)t}\leq |\widehat{\dot{\Delta}_{q}f}|e^{-c_{2}t|\xi|^2}\lesssim |\widehat{\dot{\Delta}_{q}f}|e^{-c_{2}(2^{q}\sqrt{t})^2} \label{R-E16}
\end{eqnarray}
for $c_{2}>0$. Furthermore, we can obtain
\begin{eqnarray}
2^{q\sigma}\|\widehat{\dot{\Delta}_{q}f}e^{-\eta(\xi)t}\|_{L^{2}}
\lesssim\|f\|_{\dot{B}^{-s}_{2,\infty}}(1+t)^{-\frac{\sigma+s}{2}}[(2^{q}\sqrt{t})^{\sigma+s}e^{-c_{2}(2^{q}\sqrt{t})^2}] \label{R-E17}
\end{eqnarray}
for $\sigma\in \mathbb{R}, s\in \mathbb{R}$ such that $\sigma+s>0$. Note that
\begin{eqnarray}
\Big\|(2^{q}\sqrt{t})^{\sigma+s}e^{-c_{2}(2^{q}\sqrt{t})^2}\Big\|_{l^{r}_{q}}\lesssim 1, \label{R-E18}
\end{eqnarray}
for any $r\in [1,+\infty]$. Combining (\ref{R-E15}),(\ref{R-E17})-(\ref{R-E18}), we conclude that
\begin{eqnarray}
&&\Big\|2^{q\sigma}\|\widehat{\dot{\Delta}_{q}f}e^{-\eta(\xi)t}\|_{L^{2}}\Big\|_{l^{r}_{q}}\nonumber\\ &\lesssim & \|f\|_{\dot{B}^{-s}_{2,\infty}}(1+t)^{-\frac{\sigma+s}{2}}
+\|f\|_{\dot{B}^{\sigma+\ell}_{p,r}}(1+t)^{-\frac{\ell}{2}+\frac{n}{2}(\frac{1}{p}-\frac{1}{2})}, \label{R-E19}
\end{eqnarray}
which is just the inequality (\ref{R-E12}).
\end{proof}

\section{Global-in-time existence}\setcounter{equation}{0}\label{sec:4}
As shown by \cite{MXK}, the recent local existence theory in \cite{XK1} for generally symmetric hyperbolic systems
can be applied to (\ref{R-E4}) directly.
\begin{prop}\label{prop4.1} (\cite{MXK}) Assume that
$U_{0}\in{B^{3/2}_{2,1}}$, then there exists a time
$T_{0}>0$ (depending only on the initial data) such that
\begin{itemize}
\item[(i)] (Existence):
system (\ref{R-E4}) has a unique solution
$U(t,x)\in\mathcal{C}^{1}([0,T_{0}]\times \mathbb{R})$ satisfying
$U\in\widetilde{\mathcal{C}}_{T_{0}}(B^{3/2}_{2,1})\cap
\widetilde{\mathcal{C}}^1_{T_{0}}(B^{1/2}_{2,1})$;
\item[(ii)] (Blow-up criterion): if the maximal time $T^{*}(>T_{0})$ of existence of such a solution
is finite, then
$$\limsup_{t\rightarrow T^{*}}\|U(t,\cdot)\|_{B^{3/2}_{2,1}}=\infty$$
if and only if $$\int^{T^{*}}_{0}\|\nabla
U(t,\cdot)\|_{L^{\infty}}dt=\infty.$$
\end{itemize}
\end{prop}

Furthermore, in order to show that classical solutions in Proposition \ref{prop3.1} are globally defined,
we need to construct a priori estimates according to the dissipative mechanism produced by the Tomoshenko system.
For this purpose, define by $E(T)$ the energy functional and by $D(T)$ the corresponding dissipation functional:
$$E(T):=\|U\|_{\widetilde{L}^\infty_{T}(B^{3/2}_{2,1})}$$
and
$$D(T):=\|y\|_{\widetilde{L}^2_{T}(B^{3/2}_{2,1})}+\|(v,z_{x})\|_{\widetilde{L}^2_{T}(B^{1/2}_{2,1})}
+\|u_{x}\|_{\widetilde{L}^2_{T}(B^{-1/2}_{2,1})}$$
for any time $T>0$. Hence, we have the following

\begin{prop}\label{prop4.2} Suppose $U\in\widetilde{\mathcal{C}}_{T}(B^{3/2}_{2,1})\cap
\widetilde{\mathcal{C}}^1_{T}(B^{1/2}_{2,1})$ is a solution of
(\ref{R-E4}) for $T>0$. There exists $\delta_{1}>0$ such that if $E(T)\leq\delta_{1}, $
then
\begin{eqnarray}
&&E(T)+D(T)\lesssim \|U_{0}\|_{B^{3/2}_{2,1}}+\Big(\sqrt{E(T)}+E(T)\Big)D(T). \label{R-E20}
\end{eqnarray}
Furthermore, it holds that
\begin{eqnarray}
&&E(T)+D(T)\lesssim \|U_{0}\|_{B^{3/2}_{2,1}}. \label{R-E21}
\end{eqnarray}
\end{prop}

Actually, in the case of non-equal wave speeds ($a\neq1$), a priori estimates on the dissipations for $y,z_{x}$ and $u_{x}$ coincide with the case of equal wave speeds. For brevity, we present them as lemmas only, the interested reader is referred to \cite{MXK} for proofs.

\begin{lem}(The dissipation for $y$)\label{lem4.1}
If $U\in\widetilde{\mathcal{C}}_{T}(B^{3/2}_{2,1})\cap
\widetilde{\mathcal{C}}^1_{T}(B^{1/2}_{2,1})$ is a solution of
(\ref{R-E4}) for any $T>0$, then
\begin{eqnarray}
&&E(T)+\|y\|_{\widetilde{L}^{2}_{T}(B^{3/2}_{2,1})}\lesssim \|U_{0}\|_{B^{3/2}_{2,1}}+\sqrt{E(T)}D(T). \label{R-E22}
\end{eqnarray}
\end{lem}

\begin{lem}\label{lem4.2}
(The dissipation for $z_{x}$) If $U\in\widetilde{\mathcal{C}}_{T}(B^{3/2}_{2,1})\cap
\widetilde{\mathcal{C}}^1_{T}(B^{1/2}_{2,1})$ is a solution of
(\ref{R-E4}) for any $T>0$, then
\begin{eqnarray}
\|z_{x}\|_{\widetilde{L}^{2}_{T}(B^{1/2}_{2,1})} &\lesssim & E(T)+\|U_{0}\|_{B^{3/2}_{2,1}}+\|y\|_{\widetilde{L}^{2}_{T}(B^{3/2}_{2,1})}\nonumber\\ &&+
\|v\|_{\widetilde{L}^{2}_{T}(B^{1/2}_{2,1})}+\sqrt{E(T)}D(T). \label{R-E23}
\end{eqnarray}
\end{lem}

\begin{lem}\label{lem4.3}
(The dissipation for $u_{x}$) If $U\in\widetilde{\mathcal{C}}_{T}(B^{3/2}_{2,1})\cap
\widetilde{\mathcal{C}}^1_{T}(B^{1/2}_{2,1})$ is a solution of
(\ref{R-E4}) for any $T>0$, then
\begin{eqnarray}
\|u_{x}\|_{\widetilde{L}^{2}_{T}(B^{-1/2}_{2,1})}\lesssim E(T)+\|U_{0}\|_{B^{3/2}_{2,1}}+\|v\|_{\widetilde{L}^{2}_{T}(B^{1/2}_{2,1})}+\|y\|_{\widetilde{L}^{2}_{T}(B^{3/2}_{2,1})}. \label{R-E24}
\end{eqnarray}
\end{lem}

However, the calculation for the dissipation of $v$ is a little different. We would like to give the proof as follows.
\begin{lem}(The dissipation for $v$)\label{lem4.4} If $U\in\widetilde{\mathcal{C}}_{T}(B^{3/2}_{2,1})\cap
\widetilde{\mathcal{C}}^1_{T}(B^{1/2}_{2,1})$ is a solution of
(\ref{R-E4}) for any $T>0$, then
\begin{eqnarray}
\|v\|_{\widetilde{L}^{2}_{T}(B^{1/2}_{2,1})}&\lesssim& E(T)+\|U_{0}\|_{B^{3/2}_{2,1}}+\varepsilon\|u_{x}\|_{\widetilde{L}^{2}_{T}(B^{-1/2}_{2,1})}
\nonumber\\&&+(1+C_{\varepsilon})\|y\|_{\widetilde{L}^{2}_{T}(B^{3/2}_{2,1})}
+E(T)D(T) \label{R-E25}
\end{eqnarray}
for $\varepsilon>0$, where $C_{\varepsilon}$ is a position constant dependent on $\varepsilon$.
\end{lem}

\begin{proof}
It is convenient to rewrite the system (\ref{R-E4}) as follows:
\begin{align}\label{R-E26}
\left\{\begin{array}{l}
        v_t-u_x+y=0,\\[2mm]
        u_t-v_x=0,\\[2mm]
        z_t-ay_x=0,\\[2mm]
         y_t-az_{x}-v+\gamma y=g(z)_{x},
\end{array}\right.
\end{align}
where the smooth function $g(z)$ is defined by
$$g(z)=\sigma(z/a)-\sigma(0)-\sigma'(0)z/a=O(z^2)$$ satisfying $g(0)=0$ and $g'(0)=0$.

Firstly, applying the inhomogeneous frequency-localization operator $\Delta_{q}(q\geq-1)$ to (\ref{R-E26}) gives
\begin{align}\label{R-E27}
\left\{\begin{array}{l}
        \Delta_{q}v_t-\Delta_{q}u_x+\Delta_{q}y=0,\\[2mm]
        \Delta_{q}u_t-\Delta_{q}v_x=0,\\[2mm]
        \Delta_{q}z_t-a\Delta_{q}y_x=0,\\ [2mm]
        \Delta_{q}y_t-a\Delta_{q}z_{x}-\Delta_{q}v+\gamma \Delta_{q}y=\Delta_{q}g(z)_{x}.
\end{array}\right.
\end{align}
Next, multiplying the first equation in (\ref{R-E27}) by $-\Delta_{q}y$, the second one by $-a\Delta_{q}z$,
the third one by $-a\Delta_{q}u$ and the fourth one by $-\Delta_{q}v$, respectively, then adding the resulting equalities,
we have
\begin{eqnarray}
&&-(\Delta_{q}v\Delta_{q}y+a\Delta_{q}u\Delta_{q}z)_{t} +(a\Delta_{q}v\Delta_{q}z+a^2\Delta_{q}u\Delta_{q}y)_{x}+|\Delta_{q}v|^2\nonumber\\ &=&|\Delta_{q}y|^2+
(a^2-1)\Delta_{q}y\Delta_{q}u_x+\gamma\Delta_{q}y\Delta_{q}v-\Delta_{q}g(z)_{x}\Delta_{q}v. \label{R-E28}
\end{eqnarray}

Integrating the equality (\ref{R-E28}) in $x\in \mathbb{R}$, with the aid of Cauchy-Schwarz inequality, we obtain
\begin{eqnarray}
&&\frac{d}{dt}E_{1}[\Delta_{q}U]+\frac{1}{2}\|\Delta_{q}v\|^2_{L^2}\nonumber\\ &\lesssim& \|\Delta_{q}y\|^2_{L^2}+|a^2-1|\|\Delta_{q}y\|_{L^2}\|\Delta_{q}u_{x}\|_{L^2}
\nonumber\\ &&+\|\Delta_{q}g(z)_{x}\|_{L^2}\|\Delta_{q}v\|_{L^2}, \label{R-E29}
\end{eqnarray}
where
$$E_{1}[\Delta_{q}U]:=-\int_{\mathbb{R}}(\Delta_{q}v\Delta_{q}y+\Delta_{q}u\Delta_{q}z)dx.$$
By performing the integral with respect to $t\in [0,T]$, we are led to
\begin{eqnarray}
&&\|\Delta_{q}v\|^2_{L^2_{t}(L^2)}
\nonumber\\  &\lesssim &
\|\Delta_{q}U\|^2_{L^\infty_{T}(L^2)}+\|\Delta_{q}U_{0}\|^2_{L^2}+\|\Delta_{q}y\|^2_{L^2_{T}(L^2)}\nonumber\\ &&
+\|\Delta_{q}y\|_{L^2_{T}(L^2)}\|\Delta_{q}u_{x}\|_{L^2_{T}(L^2)}+\|\Delta_{q}g(z)_{x}\|^2_{L^2_{T}(L^2)}, \label{R-E30}
\end{eqnarray}
where we have noticed the case of $a\neq1$. Furthermore, Young's inequality enables us to get
\begin{eqnarray}
&&2^{\frac{q}{2}}\|\Delta_{q}v\|_{L^2_{T}(L^2)}\nonumber\\ &\lesssim& c_{q}\|U\|_{\widetilde{L}^{\infty}_{T}(B^{1/2}_{2,1})}+c_{q}\|U_{0}\|_{B^{1/2}_{2,1}}+\varepsilon c_{q}\|u_{x}\|_{\widetilde{L}^{2}_{T}(B^{-1/2}_{2,1})}
\nonumber\\ &&
+c_{q}(1+C_{\varepsilon})\|y\|_{\widetilde{L}^{2}_{T}(B^{3/2}_{2,1})}
+c_{q}\|g(z)_{x}\|_{\widetilde{L}^{2}_{T}(B^{1/2}_{2,1})}\label{R-E31}
\end{eqnarray}
for $\varepsilon>0$, where $C_{\varepsilon}$ is a position constant dependent on $\varepsilon$ and each $\{c_{q}\}$ has a possibly different form in (\ref{R-E31}), however, the bound $\|c_{q}\|_{\ell^{1}}\leq1$ is well satisfied.

Recalling the fact $g'(0)=0$, it follows from Propositions \ref{prop2.4}-\ref{prop2.5} that
\begin{eqnarray}
\|g(z)_{x}\|_{\widetilde{L}^{2}_{T}(B^{1/2}_{2,1})}&=&\|g'(z)z_{x}\|_{\widetilde{L}^{2}_{T}(B^{1/2}_{2,1})}
\nonumber\\ &\lesssim&\|g'(z)-g'(0)\|_{\widetilde{L}^{\infty}_{T}(B^{1/2}_{2,1})}\|z_{x}\|_{\widetilde{L}^{2}_{T}(B^{1/2}_{2,1})}
\nonumber\\ &\lesssim&\|z\|_{\widetilde{L}^{\infty}_{T}(B^{1/2}_{2,1})}\|z_{x}\|_{\widetilde{L}^{2}_{T}(B^{1/2}_{2,1})}.\label{R-E32}
\end{eqnarray}
Hence, together with (\ref{R-E31})-(\ref{R-E32}), by summing up  on $q\geq-1$, we deduce that
\begin{eqnarray}
&&\|v\|_{\widetilde{L}^{2}_{T}(B^{1/2}_{2,1})}\nonumber\\ &\lesssim& \|U\|_{\widetilde{L}^{\infty}_{T}(B^{1/2}_{2,1})}+\|U_{0}\|_{B^{1/2}_{2,1}}+\varepsilon \|u_{x}\|_{\widetilde{L}^{2}_{T}(B^{-1/2}_{2,1})}
\nonumber\\ &&
+(1+C_{\varepsilon})\|y\|_{\widetilde{L}^{2}_{T}(B^{3/2}_{2,1})}
+\|z\|_{\widetilde{L}^{\infty}_{T}(B^{1/2}_{2,1})}\|z_{x}\|_{\widetilde{L}^{2}_{T}(B^{1/2}_{2,1})}, \label{R-E33}
\end{eqnarray}
which leads to the inequality (\ref{R-E25}) immediately.
\end{proof}

Having Lemmas \ref{lem4.1}-\ref{lem4.4}, by taking sufficiently small $\varepsilon>0$, we can achieve the proof of Proposition \ref{prop4.2}. For brevity, we feel free to skip
the details. Furthermore, along with local existence result (Proposition \ref{prop4.1}) and a
priori estimate (Proposition \ref{prop4.2}), Theorem \ref{thm1.1} follows from the standard boot-strap argument directly, see \cite{MXK} for similar details.

\section{Optimal decay rates}\setcounter{equation}{0}\label{sec:5}
Due to the better dissipative structure in the case of $a=1$ (see \cite{MXK}), we performed the Littlewood-Paley pointwise estimates for the linearized problem (\ref{R-E5}) and develop decay properties in the framework of Besov spaces. Furthermore, with the help of the frequency-localization Duhamel principle,
the optimal decay estimates of (\ref{R-E4}) are shown by localized time-weighted energy approaches.
For the case of $a\neq1$, if the standard Duhamel principle is used, we need to deal with the weak mechanism of regularity-loss in the price of extra higher regularity,
so it is impossible to achieve $s_{D}=3/2$. Hence, we involve new observations. Actually, we perform ``the square formula of the Duhamel principle" based on the Littlewood-Paley pointwise estimate in Fourier space for the linear system with right-hand side, see (\ref{R-E36})-(\ref{R-E37}). Furthermore, we  proceed the optimal decay estimate for (\ref{R-E4}) in terms of high-frequency and low-frequency decompositions, with the aid of the frequency-localization time-decay inequality developed in Sect.\ref{sec:3}.

To do this, we define the following energy functionals:
\begin{equation}
\mathcal{N}(t)=\sup_{0\leq\tau \leq t}(1+\tau)^{\frac{1}{4}}\|U(\tau)\|_{L^{2}},\ \ \ \mathcal{D}(t)=\|z_{x}(\tau)\|_{L^2_{t}(\dot{B}^{1/2}_{2,1})}. \nonumber
\end{equation}
The optimal decay estimate lies in a nonlinear time-weighted energy inequality, which is include in the following
\begin{lem}\label{lem5.1}
Let $U=(v,u,z,y)^{\top}$ be the global classical solutions in Theorem \ref{thm1.1}. Additionally, if $U_{0}\in\dot{B}^{-1/2}_{2,\infty}$, then it holds that
\begin{eqnarray}
\mathcal{N}(t)\lesssim \|U_{0}\|_{B^{3/2}_{2,1}\cap \dot{B}^{-1/2}_{2,\infty}}+\mathcal{N}(t)\mathcal{D}(t)+\mathcal{N}(t)^2. \label{R-E34}
\end{eqnarray}
\end{lem}

\begin{proof}
As in \cite{MK}, perform the energy method in Fourier spaces to get
\begin{eqnarray}
\frac{d}{dt}E[\hat{U}]+c_{3}\eta_{1}(\xi)|\hat{U}|^2\lesssim \xi^2|\hat{g}|^2, \label{R-E35}
\end{eqnarray}
with $\eta_{1}(\xi)=\frac{\xi^2}{(1+\xi^2)^2}$, where $E[\hat{U}]\approx|\hat{U}|^2$.
As a matter of fact, following from the derivation of (\ref{R-E35}), we can obtain the
corresponding Littlewood-Paley pointwise energy inequality
\begin{eqnarray}
\frac{d}{dt}E[\widehat{\dot{\Delta}_{q}U}]+c_{3}\eta_{1}|\widehat{\dot{\Delta}_{q}U}|^2\lesssim \xi^2|\widehat{\dot{\Delta}_{q}g}|^2, \label{R-E36}
\end{eqnarray}
where $E[\widehat{\dot{\Delta}_{q}U}]\approx|\widehat{\dot{\Delta}_{q}U}|^2$. Gronwall's inequality implies that
\begin{eqnarray}
|\widehat{\dot{\Delta}_{q}U}|^2\lesssim e^{-c_{3}\eta_{1}t}|\widehat{\dot{\Delta}_{q}U_{0}}|^2+\int_{0}^{t}e^{-c_{3}\eta_{1}(t-\tau)}\xi^2|\widehat{\dot{\Delta}_{q}g}|^2d\tau. \label{R-E37}
\end{eqnarray}
It follows from Fubini and Plancherel theorems that
\begin{eqnarray}
\|U\|^2_{L^2}&=&\sum_{q\in \mathbb{Z}}\|\dot{\Delta}_{q}U\|^{2}_{L^{2}}\nonumber\\&\lesssim&
\sum_{q\in \mathbb{Z}}\|\widehat{\dot{\Delta}_{q}U_{0}}e^{-\frac{1}{2}c_{3}\eta_{1}(\xi)t}\|^{2}_{L^{2}}\nonumber\\&&+\int^{t}_{0}
\sum_{q\in \mathbb{Z}}\||\xi|\widehat{\dot{\Delta}_{q}g}e^{-\frac{1}{2}c_{3}\eta_{1}(\xi)(t-\tau)}\|^{2}_{L^{2}}d\tau
\nonumber\\&\triangleq& I_{1}+I_{2}. \label{R-E38}
\end{eqnarray}
For $I_{1}$, by taking $p=r=2, \sigma=0, s=1/2$ and $\ell=1$ in Proposition \ref{prop3.1}, we arrive at
\begin{eqnarray}
I_{1}&=&\Big(\sum_{q<0}+\sum_{q\geq0}\Big)\Big(\cdot\cdot\cdot\Big)\nonumber
\\&\lesssim&\|U_{0}\|^{2}_{\dot{B}^{-1/2}_{2,\infty}}(1+t)^{-\frac{1}{2}}+\sum_{q\geq0}2^{2q}\|\dot{\Delta}_{q}U_{0}\|^{2}_{L^{2}}(1+t)^{-1}\nonumber
\\&\lesssim&\|U_{0}\|^{2}_{\dot{B}^{-1/2}_{2,\infty}}(1+t)^{-\frac{1}{2}}+\|U_{0}\|^{2}_{\dot{B}^{1}_{2,2}}(1+t)^{-1}\nonumber
\\&\lesssim&\|U_{0}\|^{2}_{\dot{B}^{-1/2}_{2,\infty}\cap B^{3/2}_{2,1}}(1+t)^{-\frac{1}{2}}. \label{R-E39}
\end{eqnarray}

Next, we begin to bound the nonlinear term on the right-hand side of (\ref{R-E38}), which is written as the sum of low-frequency and high-frequency
\begin{eqnarray}
I_{2}=\int^{t}_{0}\Big(\sum_{q<0}+\sum_{q\geq0}\Big)\Big(\cdot\cdot\cdot\Big)\triangleq I_{2L}+I_{2H}. \label{R-E40}
\end{eqnarray}
For $I_{2L}$, by taking $r=2, \sigma=1$ and $ s=1/2$ in Proposition \ref{prop3.1}, we have
\begin{eqnarray}
I_{2L}&&\leq\int^{t}_{0}(1+t-\tau)^{-\frac{3}{2}}\|g(z)\|^{2}_{\dot{B}^{-1/2}_{2,\infty}}d\tau\nonumber\\&&
\lesssim\int^{t}_{0}(1+t-\tau)^{-\frac{3}{2}}\|g(z)\|^{2}_{L^{1}}d\tau\nonumber\\&&\lesssim
\int^{t}_{0}(1+t-\tau)^{-\frac{3}{2}}\|z(\tau)\|^4_{L^2}d\tau\nonumber\\&&\lesssim
\mathcal{N}^{4}(t)\int^{t}_{0}(1+t-\tau)^{-\frac{3}{2}}(1+t)^{-1}d\tau\nonumber\\&&\lesssim \mathcal{N}^{4}(t)(1+t)^{-1}, \label{R-E41}
\end{eqnarray}
where we used the embedding $L^1(\mathbb{R})\hookrightarrow \dot{B}^{-1/2}_{2,\infty}(\mathbb{R})$ in Lemma \ref{lem2.3} and the fact $g(z)=O(z^2)$.
For the high-frequency part $I_{2H}$, more elaborate estimates are needed. For the purpose, we write
\begin{eqnarray*}
I_{2H}=\Big(\int^{t/2}_{0}+\int^{t}_{t/2}\Big)\Big(\cdot\cdot\cdot\Big)\triangleq I_{2H1}+I_{2H2}.
\end{eqnarray*}
For $I_{2H1}$, taking $p=r=2, \sigma=1$ and $\ell=1/2$ in Proposition \ref{prop3.1} gives
\begin{eqnarray}
I_{2H1}&=&\int_{0}^{t/2}\sum_{q\geq0}2^{3q}\|\dot{\Delta}_{q}g(z)\|^{2}_{L^{2}}(1+t-\tau)^{-\frac{1}{2}}d\tau\nonumber
\\&\leq&\int_{0}^{t/2}(1+t-\tau)^{-\frac{1}{2}}\|g(z)\|^2_{\dot{B}^{3/2}_{2,2}}d\tau. \label{R-E411}
\end{eqnarray}
On the other hand, recalling $g(z)=O(z^2)$, Proposition \ref{prop2.1} and Lemmas \ref{lem2.1}-\ref{lem2.2} enable us to get
\begin{eqnarray}
\|g(z)\|_{\dot{B}^{3/2}_{2,2}}\lesssim\|g(z)\|_{\dot{B}^{3/2}_{2,1}}\lesssim \|z\|_{L^\infty}\|z_{x}\|_{\dot{B}^{1/2}_{2,1}}. \label{R-E412}
\end{eqnarray}
Combine (\ref{R-E411}) and (\ref{R-E412}) to arrive at
\begin{eqnarray}
I_{2H1}&\lesssim&\int_{0}^{t/2}(1+t-\tau)^{-\frac{1}{2}}\|z(\tau)\|^2_{L^\infty}\|z_{x}(\tau)\|^2_{\dot{B}^{1/2}_{2,1}}d\tau
\nonumber \\&\lesssim& \sup_{0\leq\tau\leq t/2}\Big\{(1+t-\tau)^{-\frac{1}{2}}\|z(\tau)\|^{2}_{L^{\infty}}\Big\}\int_{0}^{t/2}\|z_{x}(\tau)\|^2_{\dot{B}^{1/2}_{2,1}}d\tau
\nonumber \\&\lesssim& (1+t)^{-\frac{1}{2}}\|U_{0}\|^2_{B^{3/2}_{2,1}}\mathcal{D}^2(t)
\nonumber \\&\lesssim& (1+t)^{-\frac{1}{2}}\|U_{0}\|^{2}_{B^{3/2}_{2,1}}. \label{R-E42}
\end{eqnarray}
For the last step of (\ref{R-E42}), we would like to explain a little. It follows from Proposition \ref{prop2.3} and Remark \ref{rem6.1} that
\begin{eqnarray}
\mathcal{D}(t)\lesssim \|z_{x}\|_{\widetilde{L}^2_{t}(\dot{B}^{1/2}_{2,1})}\lesssim \|z_{x}\|_{\widetilde{L}^2_{t}(B^{1/2}_{2,1})} \lesssim \|U_{0}\|_{B^{3/2}_{2,1}},\label{R-E499}
\end{eqnarray}
where we used the energy inequality (\ref{R-E9}) in Theorem \ref{thm1.1}. By choosing $r=2, p=\sigma=1$ and $\ell=1/2$ in Proposition \ref{prop3.1}, $I_{2H2}$ is proceeded as
\begin{eqnarray}
I_{2H2}&=&\int_{t/2}^{t}\sum_{q\geq0}2^{3q}\|\dot{\Delta}_{q}g(z)\|^{2}_{L^{1}}d\tau\nonumber \\&\leq& \int_{t/2}^{t}\|g(z)\|^2_{\dot{B}^{3/2}_{1,2}}d\tau. \label{R-E43}
\end{eqnarray}
Thanks to $g(z)=O(z^2)$, it follows from Proposition \ref{prop2.2} that
\begin{eqnarray}
\|g(z)\|_{\dot{B}^{3/2}_{1,2}}\leq\|g(z)\|_{\dot{B}^{3/2}_{1,1}}\lesssim\|z\|_{L^{2}}\|z_{x}\|_{\dot{B}^{1/2}_{2,1}}. \label{R-E44}
\end{eqnarray}
Together with (\ref{R-E43})-(\ref{R-E44}), we are led to
\begin{eqnarray}
I_{2H2}&\lesssim& \mathcal{N}^{2}(t)\int_{t/2}^{t} (1+\tau)^{-\frac{1}{2}}\|z_{x}(\tau)\|^2_{\dot{B}^{1/2}_{2,1}}d\tau \nonumber
\\&\lesssim&\mathcal{N}^{2}(t) \sup_{t/2\leq\tau\leq t}(1+\tau)^{-\frac{1}{2}} \int_{t/2}^{t}\|z_{x}(\tau)\|^2_{\dot{B}^{1/2}_{2,1}}d\tau
 \nonumber
\\&\lesssim& (1+t)^{-\frac{1}{2}} \mathcal{N}^{2}(t)\mathcal{D}^2(t). \label{R-E45}
\end{eqnarray}
Combine (\ref{R-E42}) and (\ref{R-E45}) to get
\begin{eqnarray}
I_{2H}&\lesssim& (1+t)^{-\frac{1}{2}}\|U_{0}\|^{2}_{B^{3/2}_{2,1}}+(1+t)^{-\frac{1}{2}} \mathcal{N}^{2}(t)\mathcal{D}^2(t). \label{R-E46}
\end{eqnarray}
Therefore, it follows from (\ref{R-E41}) and (\ref{R-E46}) that
\begin{eqnarray}
I_{2}&\lesssim&(1+t)^{-1}\mathcal{N}^{4}(t)+(1+t)^{-\frac{1}{2}}\|U_{0}\|^{2}_{B^{3/2}_{2,1}}\nonumber
\\&&
+(1+t)^{-\frac{1}{2}} \mathcal{N}^{2}(t)\mathcal{D}^2(t). \label{R-E47}
\end{eqnarray}
Finally, noticing (\ref{R-E38})-(\ref{R-E39}) and (\ref{R-E47}), we conclude that
\begin{eqnarray}
\|U\|^2_{L^2}&\lesssim& (1+t)^{-\frac{1}{2}}\|U_{0}\|^{2}_{\dot{B}^{-1/2}_{2,\infty}\cap B^{3/2}_{2,1}}+(1+t)^{-\frac{1}{2}} \mathcal{N}^{2}(t)\mathcal{D}^2(t)\nonumber
\\&&+(1+t)^{-1}\mathcal{N}^{4}(t) \label{R-E48}
\end{eqnarray}
which leads to (\ref{R-E34}) directly.
\end{proof}

\noindent \textbf{Proof of Theorem \ref{thm1.2}.}
Note that (\ref{R-E499}), we arrive at
\begin{eqnarray}
\mathcal{D}(t)\lesssim\|U_{0}\|_{B^{3/2}_{2,1}}\lesssim \|U_{0}\|_{B^{3/2}_{2,1}\cap\dot{B}^{-1/2}_{2,\infty}}. \label{R-E50}
\end{eqnarray}
Thus, if the norm $\|U_{0}\|_{B^{3/2}_{2,1}\cap\dot{B}^{-1/2}_{2,\infty}}$ is sufficiently small, then we have
\begin{eqnarray}
\mathcal{N}(t)\lesssim \|U_{0}\|_{B^{3/2}_{2,1}\cap\dot{B}^{-1/2}_{2,\infty}}+\mathcal{N}(t)^{2} \label{R-E51}
\end{eqnarray}
which implies that $\mathcal{N}(t)\lesssim \|U_{0}\|_{B^{3/2}_{2,1}\cap\dot{B}^{-1/2}_{2,\infty}}$, provided that
$\|U_{0}\|_{B^{3/2}_{2,1}\cap\dot{B}^{-1/2}_{2,\infty}}$ is sufficiently small. Consequently,
the desired decay estimate in Theorem \ref{thm1.2} follows
\begin{eqnarray}
\|U\|_{L^{2}}\lesssim \|U_{0}\|_{B^{3/2}_{2,1}\cap\dot{B}^{-1/2}_{2,\infty}}(1+t)^{-\frac{1}{4}}. \label{R-E52}
\end{eqnarray}
Hence, the proof of Theorem \ref{thm1.2} is complete eventually. \hspace{32mm} $\square$

\section{Appendix}\setcounter{equation}{0}\label{sec:6}
For convenience of reader, in this section, we review the Littlewood--Paley
decomposition and definitions for Besov spaces and Chemin-Lerner spaces in $\mathbb{R}^{n}(n\geq1)$, see
\cite{BCD} for more details.

Let ($\varphi, \chi)$ is a
couple of smooth functions valued in [0,1] such that $\varphi$ is
supported in the shell
$\textbf{C}(0,\frac{3}{4},\frac{8}{3})=\{\xi\in\mathbb{R}^{n}|\frac{3}{4}\leq|\xi|\leq\frac{8}{3}\}$,
$\chi$ is supported in the ball $\textbf{B}(0,\frac{4}{3})=
\{\xi\in\mathbb{R}^{n}||\xi|\leq\frac{4}{3}\}$ satisfying
$$
\chi(\xi)+\sum_{q\in\mathbb{N}}\varphi(2^{-q}\xi)=1,\ \ \ \ q\in
\mathbb{N},\ \ \xi\in\mathbb{R}^{n}
$$
and
$$
\sum_{k\in\mathbb{Z}}\varphi(2^{-k}\xi)=1,\ \ \ \ k\in \mathbb{Z},\
\ \xi\in\mathbb{R}^{n}\setminus\{0\}.
$$
For $f\in\mathcal{S'}$(the set of temperate distributions
which is the dual of the Schwarz class $\mathcal{S}$),  define
$$
\Delta_{-1}f:=\chi(D)f=\mathcal{F}^{-1}(\chi(\xi)\mathcal{F}f),\
\Delta_{q}f:=0 \ \  \mbox{for}\ \  q\leq-2;
$$
$$
\Delta_{q}f:=\varphi(2^{-q}D)f=\mathcal{F}^{-1}(\varphi(2^{-q}|\xi|)\mathcal{F}f)\
\  \mbox{for}\ \  q\geq0;
$$
$$
\dot{\Delta}_{q}f:=\varphi(2^{-q}D)f=\mathcal{F}^{-1}(\varphi(2^{-q}|\xi|)\mathcal{F}f)\
\  \mbox{for}\ \  q\in\mathbb{Z},
$$
where $\mathcal{F}f$, $\mathcal{F}^{-1}f$ represent the Fourier
transform and the inverse Fourier transform on $f$, respectively. Observe that
the operator $\dot{\Delta}_{q}$ coincides with $\Delta_{q}$ for $q\geq0$.

Denote by $\mathcal{S}'_{0}:=\mathcal{S}'/\mathcal{P}$ the tempered
distributions modulo polynomials $\mathcal{P}$. We first give the definition of homogeneous Besov spaces.

\begin{defn}\label{defn6.1}
For $s\in \mathbb{R}$ and $1\leq p,r\leq\infty,$ the homogeneous
Besov spaces $\dot{B}^{s}_{p,r}$ is defined by
$$\dot{B}^{s}_{p,r}=\{f\in S'_{0}:\|f\|_{\dot{B}^{s}_{p,r}}<\infty\},$$
where
$$\|f\|_{\dot{B}^{s}_{p,r}}
=\left\{\begin{array}{l}\Big(\sum_{q\in\mathbb{Z}}(2^{qs}\|\dot{\Delta}_{q}f\|_{L^p})^{r}\Big)^{1/r},\
\ r<\infty, \\ \sup_{q\in\mathbb{Z}}
2^{qs}\|\dot{\Delta}_{q}f\|_{L^p},\ \ r=\infty.\end{array}\right.
$$\end{defn}

Similarly, the definition of inhomogeneous Besov spaces is stated as follows.
\begin{defn}\label{defn6.2}
For $s\in \mathbb{R}$ and $1\leq p,r\leq\infty,$ the inhomogeneous
Besov spaces $B^{s}_{p,r}$ is defined by
$$B^{s}_{p,r}=\{f\in S':\|f\|_{B^{s}_{p,r}}<\infty\},$$
where
$$\|f\|_{B^{s}_{p,r}}
=\left\{\begin{array}{l}\Big(\sum_{q=-1}^{\infty}(2^{qs}\|\Delta_{q}f\|_{L^p})^{r}\Big)^{1/r},\
\ r<\infty, \\ \sup_{q\geq-1} 2^{qs}\|\Delta_{q}f\|_{L^p},\ \
r=\infty.\end{array}\right.
$$
\end{defn}

On the other hand, we also present the definition of Chemin-Lerner
spaces first initialled by J.-Y. Chemin and N. Lerner
\cite{CL}, which are the refinement of the space-time mixed spaces
$L^{\theta}_{T}(\dot{B}^{s}_{p,r})$ or
$L^{\theta}_{T}(B^{s}_{p,r})$.

\begin{defn}\label{defn6.3}
For $T>0, s\in\mathbb{R}, 1\leq r,\theta\leq\infty$, the homogeneous
mixed Chemin-Lerner spaces
$\widetilde{L}^{\theta}_{T}(\dot{B}^{s}_{p,r})$ is defined by
$$\widetilde{L}^{\theta}_{T}(\dot{B}^{s}_{p,r}):
=\{f\in
L^{\theta}(0,T;\mathcal{S}'_{0}):\|f\|_{\widetilde{L}^{\theta}_{T}(\dot{B}^{s}_{p,r})}<+\infty\},$$
where
$$\|f\|_{\widetilde{L}^{\theta}_{T}(\dot{B}^{s}_{p,r})}:=\Big(\sum_{q\in\mathbb{Z}}(2^{qs}\|\dot{\Delta}_{q}f\|_{L^{\theta}_{T}(L^{p})})^{r}\Big)^{\frac{1}{r}}$$
with the usual convention if $r=\infty$.
\end{defn}

\begin{defn}\label{defn6.4}
For $T>0, s\in\mathbb{R}, 1\leq r,\theta\leq\infty$, the
inhomogeneous Chemin-Lerner spaces
$\widetilde{L}^{\theta}_{T}(B^{s}_{p,r})$ is defined by
$$\widetilde{L}^{\theta}_{T}(B^{s}_{p,r}):
=\{f\in
L^{\theta}(0,T;\mathcal{S}'):\|f\|_{\widetilde{L}^{\theta}_{T}(B^{s}_{p,r})}<+\infty\},$$
where
$$\|f\|_{\widetilde{L}^{\theta}_{T}(B^{s}_{p,r})}:=\Big(\sum_{q\geq-1}(2^{qs}\|\Delta_{q}f\|_{L^{\theta}_{T}(L^{p})})^{r}\Big)^{\frac{1}{r}}$$
with the usual convention if $r=\infty$.
\end{defn}

We further define
$$\widetilde{\mathcal{C}}_{T}(B^{s}_{p,r}):=\widetilde{L}^{\infty}_{T}(B^{s}_{p,r})\cap\mathcal{C}([0,T],B^{s}_{p,r})
$$ and $$\widetilde{\mathcal{C}}^1_{T}(B^{s}_{p,r}):=\{f\in\mathcal{C}^1([0,T],B^{s}_{p,r})|\partial_{t}f\in\widetilde{L}^{\infty}_{T}(B^{s}_{p,r})\},$$
where the index $T$ will be omitted when $T=+\infty$.

By Minkowski's inequality, Chemin-Lerner spaces can be linked with the usual space-time mixed spaces
$L^{\theta}_{T}(X)$ with $X=B^{s}_{p,r}$ or $\dot{B}^{s}_{p,r}$.
\begin{rem}\label{rem6.1}
It holds that
$$\|f\|_{\widetilde{L}^{\theta}_{T}(X)}\leq\|f\|_{L^{\theta}_{T}(X)}\,\,\,
\mbox{if}\,\, r\geq\theta;\ \ \ \
\|f\|_{\widetilde{L}^{\theta}_{T}(X)}\geq\|f\|_{L^{\theta}_{T}(X)}\,\,\,
\mbox{if}\,\, r\leq\theta.
$$\end{rem}

\section*{Acknowledgments}
J. Xu is partially supported by the National
Natural Science Foundation of China (11471158), the Program for New Century Excellent
Talents in University (NCET-13-0857) and the Fundamental Research Funds for the Central
Universities (NE2015005). The work is also partially supported by
Grant-in-Aid for Scientific Researches (S) 25220702 and (A) 22244009.

\end{document}